\documentclass[12pt,a4paper,reqno,oneside,nosumlimits]{amsart}
\usepackage[mathcal]{eucal}
\usepackage{amssymb}
\usepackage[nice]{nicefrac}
\usepackage{hyperref}
\usepackage{color}
\usepackage[normalem]{ulem}

\theoremstyle{definition}
\newtheorem{definition}{Definition}[section]

\newtheorem{beispiel}[definition]{Example}

  {\popQED\end{beispiel}}

\theoremstyle{plain}
\newtheorem{theorem}[definition]{Theorem}
\newtheorem{lemma}[definition]{Lemma}
\newtheorem{proposition}[definition]{Proposition}

\usepackage{vmargin}

\setpapersize{A4}
\setmargins{3cm}       
{1.5cm}                        
{14cm}                      
{23.42cm}                    
{10pt}                           
{1cm}                           
{0pt}                             
{2cm}                           

\numberwithin{equation}{section}

\DeclareMathOperator{\hdim}{hdim}
\DeclareMathOperator{\hspec}{hspec}

\DeclareMathOperator{\wt}{wt}
\DeclareMathOperator{\dens}{den}

\newcommand{\F}{\mathbb{F}}
\newcommand{\Z}{\mathbb{Z}}
\newcommand{\Q}{\mathbb{Q}}
\newcommand{\N}{\mathbb{N}}
\newcommand{\R}{\mathbb{R}}

\newcommand{\aug}{\mathfrak{a}}

\begin{document}

\title[On the Hausdorff spectra  of free pro-$p$ groups]{On the Hausdorff spectra  of free pro-$p$ groups and certain $p$-adic analytic groups}

\author[I. de las Heras]{Iker de las Heras} 
\address{Iker de las Heras:
  Department of Mathematics, University of the Basque Country UPV/EHU, 48940 Leioa, Spain}
\email{iker.delasheras@ehu.eus}

\author[B. Klopsch]{Benjamin Klopsch}
\address{Benjamin Klopsch:
  Heinrich-Heine-Universit\"at D\"usseldorf,
  Mathematisch-Naturwissenschafltiche Fakult\"at, Mathematisches Institut}
\email{klopsch@math.uni-duesseldorf.de}

\author[A. Thillaisundaram]{Anitha Thillaisundaram} 
\address{Anitha Thillaisundaram:
  Centre for Mathematical Sciences, Lund University, 223 62 Lund, Sweden}
\email{anitha.thillaisundaram@math.lu.se}

\thanks{\textit{Funding acknowledgements.} The first author was  supported by the Spanish Government, grant MTM2017-86802-P, partly  with FEDER funds, and by the Basque Government, grant IT974-16; he  received funding from the European Union's Horizon 2021 research and  innovation programme under the Marie Sklodowska-Curie grant  agreement, project 101067088. The second author received support  from the Knut and Alice Wallenberg Foundation and by the Royal Physiographic Society of Lund.  The third author  acknowledges support from EPSRC, grant EP/T005068/1 and from the  Folke Lann\'{e}r's Fund.  The research was partially conducted in  the framework of the DFG-funded research training group GRK 2240: Algebro-Geometric Methods in Algebra, Arithmetic and Topology.}
    



\begin{abstract} 
  We establish that finitely generated non-abelian direct products $G$
  of free pro\nobreakdash-$p$ groups have full Hausdorff spectrum
  $\hspec^{\mathcal{L}}(G) = [0,1]$ with respect to the
    lower $p$-series $\mathcal{L}$.  This complements similar results
    with respect to other standard filtration series and a recent
    theorem showing that
    the Hausdorff spectrum $\hspec^\mathcal{L}(G)$ of a $p$-adic
    analytic pro\nobreakdash-$p$ group~$G$ is discrete and consists of
    at most $2^{\dim(G)}$ rational numbers.

  The latter also left some room for improvement
    regarding the upper bound.  Indeed, for finitely generated
    nilpotent pro-$p$ groups $G$ we obtain the stronger assertion
    $\lvert \hspec^\mathcal{L}(G) \rvert \le \dim(G)$.  Moreover, we
    produce a corresponding result when the $p$-adic analytic
  pro\nobreakdash-$p$ group $G$ is just infinite, which holds
  not just for $\mathcal{L}$ but for arbitrary filtration
  series.
  
  Finally, we show that, if $G$ is a countably based
  pro\nobreakdash-$p$ group with an open subgroup mapping onto the
  free abelian pro\nobreakdash-$p$ group $\Z_p \oplus \Z_p$, then for
  every prescribed finite set $\{0,1\} \subseteq X \subseteq [0,1]$
  there is a filtration series~$\mathcal{S}$ such that
  $\hspec^\mathcal{S}(G) = X$; in particular,
  $\lvert \hspec^{\mathcal{S}}(G) \rvert$ is unbounded, as
  $\mathcal{S}$ runs through all filtration series of~$G$ with
  $\lvert \hspec^{\mathcal{S}}(G) \rvert < \infty$.
\end{abstract}

\maketitle


\section{Introduction}

Let $p$ denote a prime throughout. This paper concerns the Hausdorff spectra of certain infinite finitely generated
pro\nobreakdash-$p$ groups, mainly with respect to the lower
$p$-series.  Recall that for a finitely generated pro\nobreakdash-$p$
group~$G$, the \emph{lower $p$-series}
$\mathcal{L}\colon P_i(G),\, i \in \N$, -- sometimes
called the lower $p$-central series -- is defined recursively as
follows:
\begin{align*}
  P_1(G) = G,  %
  & \quad \text{and} \quad  P_i(G) = \overline{P_{i-1}(G)^p
    \, [P_{i-1}(G),G]} \quad
    \text{for $i\in \N$ with $i \ge 2$.}  
\end{align*}

\smallskip

The study of Hausdorff dimension in profinite groups is a
current though established subject; for
instance, see \cite{KlThZu19} for an overview and
\cite{KlTh19,GaGaKl20,GoZo21,HeKl22,HeTh22a,HeTh22b, FA} for some
recent developments.  We recall here a group-theoretical
description of Hausdorff dimension (cf. \cite{Ab94,BaSh97}): let $G$
be an infinite finitely generated profinite group and
let $\mathcal{S} \colon G_i,\, i \in \N_0$, be a \emph{filtration
  series} of~$G$, that is, a descending chain
$G = G_0 \ge G_1 \ge \ldots$ of open normal subgroups
$G_i \trianglelefteq_\mathrm{o} G$ such that
$\bigcap_{i \in \N_0} G_i = 1$.  The filtration series
  $\mathcal{S}$ yields a translation-invariant metric on~$G$ given by
  $d^\mathcal{S}(x,y) = \inf \left\{ \lvert G : G_i \rvert^{-1} \mid x
    \equiv y \pmod{ G_i} \right\}$, for $x,y \in G$.  This gives rise
  to the \emph{Hausdorff dimension} $\hdim_G^\mathcal{S}(U) \in [0,1]$
  of any subset $U \subseteq G$, with respect to the filtration
  series~$\mathcal{S}$. The Hausdorff dimension of a
closed subgroup $H \le_\mathrm{c} G$, with respect to~$\mathcal{S}$,
has an equivalent algebraic interpretation as
\begin{equation*}
  \hdim_G^\mathcal{S}(H) =
  \varliminf_{i\to \infty} \frac{\log \lvert HG_i : G_i
    \rvert}{\log \lvert G : G_i \rvert},
\end{equation*}
and the \emph{Hausdorff spectrum} of $G$, with respect to
$\mathcal{S}$, is
\[
  \hspec^\mathcal{S}(G) = \{ \hdim_G^\mathcal{S}(H) \mid H
  \le_\mathrm{c} G\} \subseteq [0,1].
\]
We say that $H$ has \emph{strong Hausdorff
  dimension} in $G$, with respect to $\mathcal{S}$, if
\[
  \hdim^{\mathcal{S}}_G(H) = \lim_{i \to \infty} \frac{\log \lvert H
    G_i : G_i \rvert}{\log \lvert G : G_i \rvert}
\]
is given by a proper limit.

It is known that even for well-behaved groups, such as $p$-adic
analytic pro\nobreakdash-$p$ groups~$G$, the Hausdorff dimension of
individual subgroups, and also the Hausdorff spectrum of~$G$, are
sensitive to the choice of~$\mathcal{S}$;
see~\cite[Thm.~1.3]{KlThZu19}.  For a finitely generated
pro\nobreakdash-$p$ group~$G$, the standard choices of $\mathcal{S}$
are primarily the $p$-power series~$\mathcal{P}$, the iterated
$p$-power series~$\mathcal{P}^*$, the Frattini series~$\mathcal{F}$,
the dimension subgroup (or Zassenhaus) series~$\mathcal{D}$ and the
lower $p$-series~$\mathcal{L}$.  If $G$ is a $p$-adic analytic
pro\nobreakdash-$p$ group, then from~\cite[Prop.~1.5]{KlThZu19}, the
Hausdorff dimensions of $H \le_\mathrm{c} G$ with respect to the
series $\mathcal{P}, \mathcal{F}, \mathcal{D}$ coincide,
and~\cite[Thm.~1.1]{BaSh97} yields that, for $\mathcal{S}$ any one of
these series $\mathcal{P}, \mathcal{F}, \mathcal{D}$, the spectrum
\[
  \hspec^\mathcal{S}(G) = \Big\{ \nicefrac{\displaystyle
    \dim(H)}{\displaystyle \dim(G)} \mid H \le_\mathrm{c} G \Big\}
\]
is discrete and that all closed subgroups have strong Hausdorff
dimension in $G$ with respect to $\mathcal{S}$. We note that these
results also hold with respect to the
series~$\mathcal{P}^*$. Regarding the series $\mathcal{L}$, the form
of the Hausdorff spectrum $\hspec^{\mathcal{L}}(G)$, and its
finiteness, was only recently established
in~\cite[Thm.~1.5]{HeKlTh25}.

It remains an outstanding problem in the subject to understand, more
generally, what renders the Hausdorff spectrum $\hspec^\mathcal{S}(G)$
of a finitely generated pro\nobreakdash-$p$ group $G$ finite. In
particular, it is open whether finiteness of the spectrum with respect
to one of the standard filtration series
$\mathcal{P}, \mathcal{P}^*, \mathcal{F}, \mathcal{D}, \mathcal{L}$
implies that $G$ is $p$-adic analytic; compare with
\cite[Prob.~1.4]{KlThZu19}.
It is of interest to identify and examine potential
  obstacles to such a new characterization of $p$-adic analytic
  pro\nobreakdash-$p$ groups.  For the filtration series
$\mathcal{F}$ and $\mathcal{D}$, it was shown in~\cite{GaGaKl20} that
free pro\nobreakdash-$p$ groups $F$ can be cancelled from the
candidate list of possible counterexamples; in fact,
non-abelian finite direct products $G$ of free pro\nobreakdash-$p$
groups have full Hausdorff spectrum
$\hspec^\mathcal{F}(G) = \hspec^\mathcal{D}(G) = [0,1]$. Using
Lie-theoretic tools, we produce an analogous result for the lower
$p$\nobreakdash-series~$\mathcal{L}$.

\begin{theorem} \label{thm:direct-prod-of-free-pro-p} 
  Let $G = F_1 \times \ldots \times F_r$ be a finite direct product of finitely generated free pro\nobreakdash-$p$ groups $F_j$, at least one of which is non-abelian.  Then $G$ has full Hausdorff spectrum $\hspec^{\mathcal{L}}(G) = [0,1]$ with respect to the lower $p$-series~$\mathcal{L}$.
\end{theorem}

It remains a challenge to compute the Hausdorff spectra with respect to the lower $p$-series for pro\nobreakdash-$p$ groups of positive rank gradient, such as non-soluble Demushkin pro\nobreakdash-$p$ groups.  This would further complement the results
in~\cite{GaGaKl20}.

\medskip

Next, let $G$ be an infinite $p$-adic analytic pro-$p$ group. Recall
that in~\cite[Thm.~1.5]{HeKlTh25}, it was shown that the Hausdorff
spectrum $\hspec^\mathcal{L}(G)$ is discrete and consists of at most
$2^{\dim(G)}$ rational numbers. Here we prove that under
additional assumptions on $G$, the following stronger results can be
obtained, which partly extend even to arbitrary filtration series.

\begin{theorem}\label{thm:just-infinite-or-nilpotent}
  Let $G$ be an infinite $p$-adic analytic pro\nobreakdash-$p$ group.

  \smallskip
  
  \noindent \textup{(1)} If $G$ is just infinite, then
  every closed subgroup $H \le_\mathrm{c} G$ has strong Hausdorff
  dimension $\hdim_G^\mathcal{S}(H) = \dim(H) / \dim(G)$ for every
  filtration series $\mathcal{S}$.

  \smallskip
  
  \noindent \textup{(2)} If $G$ is nilpotent, then every
  closed subgroup $H \le_\mathrm{c} G$ has strong Hausdorff dimension
  $\hdim_G^\mathcal{S}(H) = \dim(H) / \dim(G)$ for
  $\mathcal{S}=\mathcal{L}$.

  \smallskip
  
  In each of the two cases, the Hausdorff spectrum of $G$
  with respect to the relevant filtration series is equal to
  $\Big\{ \nicefrac{\displaystyle \dim(H)}{\displaystyle \dim(G)} \mid
  H \le_\mathrm{c} G \Big\}$.
\end{theorem}

Finally, referring again to \cite[Prob.~1.4]{KlThZu19},  it is known that for every $p$-adic
analytic pro\nobreakdash-$p$ group $G$ whose abelianisation has
torsion-free rank at least $2$, there is a filtration~$\mathcal{S}$
such that $\hspec^\mathcal{S}(G)$ is infinite;
see~\cite[Thm.~1.3]{KlThZu19}.  This leads to the
question~\cite[Prob.~1.6]{KlThZu19}: does there exist, for every
$p$-adic analytic pro\nobreakdash-$p$ group~$G$, a uniform bound
$b(G)$ for $\lvert \hspec^{\mathcal{S}}(G) \rvert$, as $\mathcal{S}$
runs through all filtration series of~$G$ with
$\lvert \hspec^{\mathcal{S}}(G) \rvert < \infty$.
 We show that the answer is
typically negative, in the strongest possible way.

\begin{theorem}\label{thm:uniform-bound}
  Let $G$ be a countably based pro\nobreakdash-$p$ group that has an
  open subgroup mapping surjectively onto $\Z_p \oplus \Z_p$.  Then
  for every finite subset $X \subseteq [0,1]$ with
  $\{0,1\} \subseteq X$ there exists a filtration series~$\mathcal{S}$
  such that $\hspec^{\mathcal{S}}(G) = X$.
\end{theorem}

  \noindent \textit{Notation.}  We generally explain any new notation when needed,
  sometimes implicitly.  For
  convenience we collect some basic conventions here.

The set of positive integers is denoted by $\N$ and the set
  of non-negative integers by $\N_0$.  Throughout $p$ denotes a prime, and
  $\Z_p$ is the ring of $p$-adic integers (or simply its additive
  group).    The lower limit (limes inferior) of a sequence
  $(a_i)_{i \in \N}$ in $\R \cup \{ \pm \infty \}$ is denoted by $\varliminf a_i = \varliminf_{i \to \infty} a_i$.  Intervals of real
  numbers are written as $[a,b]$, $(a,b]$ et cetera.

  The group-theoretic notation is mostly standard and in line, for
  instance, with its use in~\cite{DDMS99}.  Tacitly, subgroups of
  profinite groups are generally understood to be closed subgroups.
  In some places, we emphasise that we are taking the topological
  closure of a set $X$ by writing~$\overline{X}$.

\medskip

\noindent \textit{Organisation.} In Section~\ref{sec:free}, we prove
Theorem~\ref{thm:direct-prod-of-free-pro-p} about full Hausdorff
spectra for finite direct products of free pro-$p$ groups. In
Section~\ref{sec:just-infinite-or-nilpotent}, we prove Theorem~\ref{thm:just-infinite-or-nilpotent}, for $p$-adic analytic
pro\nobreakdash-$p$ groups that are just infinite or nilpotent. Lastly, in Section~\ref{sec:unbounded} we
establish Theorem~\ref{thm:uniform-bound}, yielding arbitrarily large
finite Hausdorff spectra.

  
\section{Free
  pro-\texorpdfstring{$p$}{p} groups} \label{sec:free}

In this section we study the Hausdorff spectra
  $\hspec^{\mathcal{L}}(G)$ of finite direct products $G$ of free
  pro-$p$ groups, with respect to the lower $p$-series~$\mathcal{L}$,
  and prove Theorem~\ref{thm:direct-prod-of-free-pro-p}.

To this end, we carefully adapt
Lie-theoretic methods used by Garaialde Oca\~na, Garrido and
Klopsch~\cite[Sec.~4]{GaGaKl20} to deal with other filtration series,
in particular the Zassenhaus series.  In parts, we follow closely a
line of reasoning from a preprint version
(\texttt{arXiv:1901.03101v2}) of~\cite{GaGaKl20} that adjusts well to
our present setting.  We slightly adjust our notation so that it
matches better with the relevant parts of~\cite{GaGaKl20}.

It is convenient to establish first the following central case of
Theorem~\ref{thm:direct-prod-of-free-pro-p}.

\begin{theorem} \label{thm:free-pro-p} Let $F$ be a finitely generated
  non-abelian free pro\nobreakdash-$p$ group.  Then $F$ has full
  Hausdorff spectrum $\hspec^{\mathcal{L}}(F) = [0,1]$ with respect to
  the lower $p$-series $\mathcal{L}$.
\end{theorem}

The proof of the theorem requires some preparation.  We consider the
free pro\nobreakdash-$p$ group $F$ on finitely many free generators
$x_1, \ldots, x_d$, for a given $d \ge 2$.  As described in
\cite[Sec.~II.1]{La65} and summarised in \cite[Sec.~5]{La85}, the
lower $p$-series $P_n(F)$, $n \in \N$, gives rise to a mixed
$\F_p$-Lie algebra $\boldsymbol{\Lambda} = (\Lambda, \mathsf{P})$.  In
concrete terms, this is the $\N$-graded $\F_p$-Lie algebra
\[
  \Lambda = \bigoplus_{n=1}^\infty \Lambda_n
\]
whose homogeneous components are the finite elementary abelian
$p$-groups $\Lambda_n = P_n(F)/P_{n+1}(F)$, for $n \in \N$, and whose
Lie bracket on homogeneous elements is induced by the group commutator
map $P_m(F) \times P_n(F) \to P_{m+n}(F)$, $(g,h) \mapsto [g,h]$;
additionally, $\Lambda$ comes equipped with a family of $\F_p$-linear
operators $\mathsf{P}_n \colon \Lambda_n \to \Lambda_{n+1}$, for
$n \in \N$, induced by the $p$-power map $P_n(F) \to P_{n+1}(F)$,
$g \mapsto g^p$ and thus satisfying a certain set of axioms that match
the more general concept of a mixed Lie algebra.

For $p>2$, the structure of $\boldsymbol{\Lambda}$ is rather easy to
describe.  Indeed, the images
$\mathsf{x}_1, \ldots, \mathsf{x}_d \in \Lambda_1$ of
$x_1, \ldots, x_d \in P_1(F)$ generate a free $\F_p$-Lie subalgebra
$L = \bigoplus_{n=1}^\infty L_n$ of~$\Lambda$, furnished with the
grading $L_n = L \cap \Lambda_n$ which reflects the degree of Lie
monomials.  The mixed Lie algebra $\boldsymbol{\Lambda}$ can be
understood as a free Lie ring on $\mathsf{x}_1, \ldots, \mathsf{x}_d$
over the polynomial ring $\F_p[\pi]$ in the sense that we may write
\begin{equation} \label{equ:lambda-n}
  \Lambda_n = \bigoplus_{m=1}^n \pi^{n-m} L_m, \qquad \text{for $n \in
    \N$,}
\end{equation}
and interpret the operators
$\mathsf{P}_n \colon \Lambda_n \to \Lambda_{n+1}$ as multiplication by
the transcendental element~$\pi$.  For $p=2$, there is a small wrinkle
for $\mathsf{P}_1 \colon \Lambda_1 \to \Lambda_2$, but the free
$\F_p$-Lie algebra $L$ still appears as a subalgebra of $\Lambda$ and
the mixed Lie subalgebra $\boldsymbol{\Lambda}^\circ$ based on
$\Lambda^\circ = [\Lambda,\Lambda] = \bigoplus_{n=2}^\infty \Lambda_n$
admits exactly the same description as for $p>2$.  It is convenient to
postpone taking this extra hurdle and for the time being we assume for
simplicity that $p > 2$; at the end of the proof of
Theorem~\ref{thm:free-pro-p} we explain how our reasoning easily
adapts to the remaining case $p=2$.

For all practical purposes, $\boldsymbol{\Lambda}$ and its $\F_p$-Lie
subalgebra $L$ can be realised as subalgebras of the free associative
algebra
$A = \F_p[\pi]\langle\!\langle \mathsf{x}_1, \ldots, \mathsf{x}_d
\rangle\!\rangle$ over the polynomial ring $\F_p[\pi]$, turned into an
$\F_p$-Lie algebra via
$[\mathsf{u},\mathsf{v}] = \mathsf{u}\mathsf{v} -
\mathsf{v}\mathsf{u}$ and equipped with a `uniform'
$\mathsf{P}$-operator $\mathsf{u} \mapsto \pi \mathsf{u}$.  Denoting
by $A_n$ the $\F_p$-subspace spanned by elements of the form
$\pi^{n-m} \mathsf{u}$, where $\mathsf{u}$ runs through all products
of length $m \in \{0,1, \ldots, n\}$ in
$\mathsf{x}_1, \ldots, \mathsf{x}_d$, we see that $A$ decomposes as
$A = \bigoplus_{n=0}^\infty A_n$ and becomes a graded algebra.  From
this perspective, the homogeneous components of $\Lambda$ arise simply
as the intersections $\Lambda_n = \Lambda \cap A_n$ with homogeneous
components of~$A$.

It is well known that each homogeneous component
$L_n = L \cap \Lambda_n$ of the free $\F_p$-Lie subalgebra $L$ on
$\mathsf{x}_1, \ldots, \mathsf{x}_d$ is spanned, as an $\F_p$-vector
space, by linearly independent \emph{basic Lie commutators} in
$\mathsf{x}_1, \ldots, \mathsf{x}_d$ of weight~$n$;
compare~\cite[Sec.~11]{Ha76} or \cite[Sec.~2]{Ba87}.  We recall and
extend the pertinent notions.

\begin{definition} \label{Defi:basic_comm_Lie} The \emph{weight} of a
  Lie commutator in $\mathsf{x}_1, \ldots, \mathsf{x}_d$ is defined
  inductively: each generator $\mathsf{x}_i$ has
  weight~$\wt(\mathsf{x}_i) = 1$, and, if $\mathsf{c}_1, \mathsf{c}_2$
  are Lie commutators in $\mathsf{x}_1, \ldots, \mathsf{x}_d$, then
  $[\mathsf{c}_1, \mathsf{c}_2]$ is a Lie commutator of weight
  $\wt([\mathsf{c}_1,\mathsf{c}_2]) = \wt(\mathsf{c}_1) +
  \wt(\mathsf{c}_2)$.  The \emph{basic commutators} are also defined
  inductively.  The basic commutators of weight $1$ are the
  generators~$\mathsf{x}_1, \ldots, \mathsf{x}_d$ in their natural
  order $\mathsf{x}_1 < \ldots < \mathsf{x}_d$.  For each $n \in \N$
  with $n \ge 2$, after defining basic commutators of weight less
  than~$n$, the basic commutators of weight~$n$ are those Lie
  commutators $[\mathsf{u},\mathsf{v}]$ such that:
  \begin{enumerate}
  \item $\mathsf{u},\mathsf{v}$ are basic commutators with
    $\wt(\mathsf{u})+\wt(\mathsf{v}) = n$,
  \item $\mathsf{u} > \mathsf{v}$; and if
    $\mathsf{u} = [\mathsf{y},\mathsf{z}]$ with basic commutators
    $\mathsf{y}, \mathsf{z}$, then $\mathsf{v} \geq \mathsf{z}$.
  \end{enumerate}
  Lastly, the total order is extended, subject to the condition that
  $\mathsf{u} < \mathsf{v}$ whenever
  $\wt(\mathsf{u}) < \wt(\mathsf{v})$ and arbitrarily among the new
  basic commutators of weight~$n$.

  We extend the terminology to elements of $\Lambda$, the free
  $\F_p[\pi]$-Lie ring on $\mathsf{x}_1, \ldots, \mathsf{x}_d$, as
  follows.  For $n \in \N$, a \emph{generalised basic commutator}
  $\tilde{\mathsf{c}}$ of weight $\wt(\tilde{\mathsf{c}}) = n$ in
  $\Lambda$ is a homogeneous element of the form
  $\tilde{\mathsf{c}} = \pi^{n-m} \mathsf{c} \in \Lambda_n$, where
  $m \in \{1,2,\ldots,n\}$ and $\mathsf{c}$ is a basic commutator of
  weight~$m$.  We refer to $\mathsf{c}$ as the \emph{core part} of the
  generalised basic commutator $\tilde{\mathsf{c}}$.
\end{definition}

For $m \in \N$, we consider the principal Lie ideals
\[
  I_m = \bigoplus_{n\geq m} \Lambda_n \trianglelefteq \Lambda \quad
  \text{and} \quad J_m = I_m \cap L = \bigoplus_{n\geq m} L_n
  \trianglelefteq L.
\]
For $\F_p$-Lie subalgebras $H$ of $\Lambda$ and $M$ of $L$ we define
the relative \emph{densities}
\begin{align*}
  \dens_\Lambda(H) %
  & = \varliminf_{n \to \infty} \frac{\dim_{\F_p} ((H +
    I_n)/I_n)}{\dim_{\F_p} (\Lambda/I_n)}, \\
  \dens_L(M) %
  & = \varliminf_{n \to \infty} \frac{\dim_{\F_p} ((M +
    J_n)/J_n)}{\dim_{\F_p} (L/J_n)};
\end{align*}
we say that a subalgebra has \emph{strong density} if its density is
given by a proper limit.

Witt's formula describes the dimensions $\dim_{\F_p}(L_n)$ of the
homogenous components of the free Lie algebra~$L$ and yields the
following asymptotic estimates for the dimensions of principal
quotients of $\Lambda$ and~$L$; compare~\cite[Lem.~4.3]{JZ08}.

\begin{lemma} \label{lem:Witt-for-mixed} In the set-up described above,
  including $d \ge 2$,
  \begin{align*}
    \dim_{\F_p}(\Lambda/I_{n+1}) %
    & = \frac{d^{n+2}}{n(d-1)^2} \big(1+o(1) \big) %
    & \text{as $n \to \infty$\phantom{.}} \\
    \intertext{and}
    \dim_{\F_p} (L/J_{n+1}) %
    & = \frac{d^{n+1}}{n(d-1)} \big(1+o(1) \big) %
    & \text{as $n \to \infty$.}
  \end{align*}
\end{lemma}

From~\cite[Thm.~1]{BaOl15} we obtain the following useful consequence.

\begin{proposition} \label{prop:fg-graded-0} Let $L$ be a non-abelian
  free $\F_p$-Lie algebra on finitely many generators.  Then every
  finitely generated proper graded Lie subalgebra $M \lneqq L$ has
  strong density $\dens_L(M) = 0$.
\end{proposition}

In order to make use of this insight, we need to transfer the
conclusion to the $\F_p[\pi]$-Lie subalgebra $H$ of $\Lambda$
generated by an $\F_p$-Lie subalgebra $M$ of~$L$.

\begin{lemma} \label{lem:fg-graded-0-transfer}
  Let $L$ be the non-abelian free $\F_p$-Lie algebra on $d$
  generators, regarded as an $\F_p$-subalgebra of the free
  $\F_p[\pi]$-Lie algebra $\Lambda$ on the same $d$ generators, as in the set-up described above.  Let $M$ be a graded $\F_p$-Lie
  subalgebra of $L$ that has strong density $\dens_L(M) = 0$ in~$L$.
  Then the $\F_p[\pi]$-Lie algebra $H$ of $\Lambda$ generated by $M$
  has strong density $\dens_{\Lambda}(H) = 0$ in $\Lambda$.
\end{lemma}

\begin{proof}
  It is convenient to write, for $n \in \N$,
  \[
    \delta_n = \frac{\dim_{\F_p} ((M + J_{n+1})/J_{n+1})}{\dim_{\F_p}
      (L/J_{n+1})} \quad \text{and} \quad \Delta_n = \frac{\dim_{\F_p}
      ((H + I_{n+1})/I_{n+1})}{\dim_{\F_p} (\Lambda/I_{n+1})}.
  \]
  We are to deduce from $\lim_{n \to \infty} \delta_n = 0$ that
  $\lim_{n \to \infty} \Delta_n = 0$.

  Let $\varepsilon \in (0,1)$, and choose $c \in \N$ such that
  $d^{-c} < \varepsilon/4$.  By Lemma \ref{lem:Witt-for-mixed} there exists
  $n_0 \in \N$ with $n_0 >c$ such that for all $n \ge n_0$,
  \begin{equation} \label{equ:eps-over-2-number-1}
    \frac{\dim_{\F_p}(\Lambda/I_{n-c+1})}{\dim_{\F_p}(\Lambda/I_{n+1})}
    \le (1+\varepsilon) d^{-c} \le (1+\varepsilon) \varepsilon/4 < \varepsilon/2
  \end{equation}
  and, because $\delta_n \to 0$ as $n \to \infty$, we can also arrange that
  \begin{equation} \label{equ:eps-over-2-number-2} \bar \delta_c (n) =
    \max \{ \delta_m \mid n-c+1 \le m \le n \} \le \varepsilon/2 \qquad
    \text{for $n \ge n_0$.}
  \end{equation}
  
  We write $M = \bigoplus_{n=1}^\infty M_n$, where
  $M_n = M \cap L_n = M \cap \Lambda_n$ for $n \in \N$.  Clearly, this
  gives $ H = \bigoplus_{n=1}^\infty H_n$, with
  $H_n = \bigoplus_{m=1}^n \pi^{n-m} M_m$ for $n \in \N$.  We deduce
  that for $n \in \N$,
  \begin{align*}
    H + I_{n+1} %
    & = \left( \bigoplus_{k=0}^{n-1} \pi^k
      \bigoplus_{m=1}^{n-k} M_m \right) \oplus I_{n+1} \\
    \intertext{and likewise} %
    \Lambda %
    & = \left( \bigoplus_{k=0}^{n-1} \pi^k
      \bigoplus_{m=1}^{n-k} L_m \right) \oplus I_{n+1}.
  \end{align*}
  This yields
  \[
    \dim_{\F_p} \big((\pi^c \Lambda + I_{n+1}) / I_{n+1} \big) =
    \sum_{k=c}^{n-1} \sum_{m=1}^{n-k} \dim_{\F_p}(L_m) = \dim_{\F_p}
    (\Lambda / I_{n-c+1} )
  \]
  and
  \begin{multline*}
    \frac{\dim_{\F_p} \big( (H + \pi^c \Lambda + I_{n+1}) / (\pi^c
      \Lambda + I_{n+1}) \big)}{\dim_{\F_p} \big(\Lambda/(\pi^c
      \Lambda + I_{n+1}) \big)} = \frac{\sum_{k=0}^{c-1}
      \sum_{m=1}^{n-k} \dim_{\F_p}(M_m)}{\sum_{k=0}^{c-1}
      \sum_{m=1}^{n-k}
      \dim_{\F_p}(L_m)} \\
    \le \max_{0 \le k \le c-1} \frac{\sum_{m=1}^{n-k}
      \dim_{\F_p}(M_m)}{\sum_{m=1}^{n-k} \dim_{\F_p}(L_m)} = \bar
    \delta_c (n).
  \end{multline*}
  Using \eqref{equ:eps-over-2-number-1} and
  \eqref{equ:eps-over-2-number-2} we conclude that for
  $n \ge n_0$,
  \begin{multline*}
    \Delta_n \le \frac{\dim_{\F_p} \big( (H + \pi^c \Lambda +
      I_{n+1})/(\pi^c \Lambda + I_{n+1}) \big) + \dim_{\F_p} \big(
      (\pi^c \Lambda + I_{n+1})/ I_{n+1} \big)}{\dim_{\F_p}
      (\Lambda/I_{n+1})}
    \\
    \le \left( 1 -
      \frac{\dim_{\F_p}(\Lambda/I_{n-c+1})}{\dim_{\F_p}(\Lambda/I_{n+1})}
    \right) \bar \delta_c (n) +
    \frac{\dim_{\F_p}(\Lambda/I_{n-c+1})}{\dim_{\F_p}(\Lambda/I_{n+1})}
    < \varepsilon/2 + \varepsilon/2 = \varepsilon.
  \end{multline*}
  Thus $\Delta_n \to 0$ as $n \to \infty$.
\end{proof}

The next result is a straightforward adaptation of
\cite[Thm~4.10]{GaGaKl20}.

\begin{theorem}\label{thm:all-densities-free-lie}
  Let $\Lambda$ be a non-abelian free $\F_p[\pi]$-Lie algebra on
  finitely many generators.  Then there exists, for each
  $\alpha \in [0,1]$, an $\F_p[\pi]$-Lie subalgebra $H \leq \Lambda$
  that is contained in $\Lambda^\circ = [\Lambda,\Lambda]$, that can be  generated by generalised basic commutators and has density~$\dens_{\Lambda}(H) = \alpha$.
\end{theorem}

\begin{proof}
  As before, let $\mathsf{x}_1, \ldots, \mathsf{x}_d$, for $d \ge 2$,
  be free generators of $\Lambda$, and recall that the
  $\F_p$-subalgebra generated by $\mathsf{x}_1, \ldots, \mathsf{x}_d$
  is a free $\F_p$-Lie algebra $L = \bigoplus_{n=1}^\infty L_n$,
  furnished with the standard grading.  We write
  $L^\circ = [L,L] = \bigoplus_{n=2}^\infty L_n$ and observe that
  \[
    \Lambda^\circ = \bigoplus_{n=2}^\infty
    \Lambda_n^\circ, \qquad \text{with} \quad
    \Lambda_n^\circ = \bigoplus_{m=2}^n \pi^{n-m}
    L_m,
  \]
  is the $\F_p[\pi]$-Lie subalgebra generated by~$L^\circ$.
  Furthermore we have
  $\dim_{\F_p}(\Lambda_n) = d+\dim_{\F_p}(\Lambda_n^\circ)$ for each
  $n \in \N$.  This implies that the sequence
  \[
    l^\circ(n) = \dim_{\F_p} \big( (\Lambda^\circ +
    I_{n+1}) / I_{n+1} \big) = \sum\nolimits_{m=2}^n
    \dim_{\F_p}(\Lambda_m^\circ), \quad n \in \N,
  \]
  is strictly increasing, grows exponentially in~$n$ and satisfies
  \begin{equation} \label{equ:l-circ-approx}
    \nicefrac{\displaystyle l^\circ(n) \,}{\displaystyle \,  \dim_{\F_p} (
      \Lambda/I_{n+1})} \to 1 \qquad \text{as $n \to \infty$;}
  \end{equation}
  see Lemma~\ref{lem:Witt-for-mixed}.  We conclude that the subalgebra
  $\Lambda^\circ$ has strong density $1$ in~$\Lambda$.  Clearly, the
  trivial subalgebra $\{0\}$ has strong density $0$ in~$\Lambda$.

  Suppose now that $\alpha\in (0,1)$.  In order to produce an
  $\F_p[\pi]$-Lie subalgebra $H \leq \Lambda^\circ$ that can be
  generated by generalised basic commutators and has density
  $\dens_{\Lambda}(H) = \alpha$  it suffices, by~\eqref{equ:l-circ-approx}, to build an
  $\F_p[\pi]$-Lie subalgebra
  $H = \bigoplus_{m = 1}^\infty H_m \leq \Lambda^\circ$, where
  $H_m = H \cap \Lambda_m^\circ = H \cap \Lambda_m$ for $m \in \N$,
  which is generated by generalised basic commutators and such that
  \begin{enumerate}
  \item
    $\alpha- \nicefrac{1}{l^\circ(n)} \leq \tfrac{1}{l^\circ(n)}
    \sum_{m=1}^n \dim_{\F_p} (H_m)$ for all $n \in \N$ with
    $n \ge 2$; and
  \item
    $\frac{1}{l^\circ(n)} \sum_{m = 1}^n \dim_{\F_p} (H_m) \leq
    \alpha$ for infinitely many $n \in \N$.
  \end{enumerate}
  We construct such a Lie subalgebra $H$ inductively as the union
  $H = \bigcup_{k=2}^\infty H(k)$ of an ascending chain of
  $\F_p[\pi]$-Lie subalgebras $H(2) \subseteq H(3) \subseteq \ldots$,
  where each term $H(k)$ is generated by a finite set
  $\widetilde{Y}_k$ of generalised basic commutators of weight at most
  $k$ and
  $\widetilde{Y}_2 \subseteq \widetilde{Y}_3 \subseteq \ldots$.

  Let $k \in \N$ with $k \ge 2$.  For $k = 2$, observe that
  $l^\circ(2) = \dim_{\F_p}(\Lambda_2^\circ)$ and pick
  $a \in \{0,1, \ldots, l^\circ(2)-1\}$ such that
  $\alpha - \nicefrac{1}{l^\circ(2)} \le \nicefrac{a}{l^\circ(2)} \le
  \alpha$.  Let $H(2)$ denote the $\F_p[\pi]$-Lie subalgebra generated
  by a subset $\widetilde{Y}_2 \subseteq \Lambda_2^\circ = L_2$,
  consisting of an arbitrary choice of $a$ basic commutators of
  weight~$2$.

  Now suppose that $k \ge 3$.  Suppose further that we have already
  constructed an $\F_p[\pi]$-Lie subalgebra $H(k-1) \le \Lambda^\circ$
  which is generated by a finite set $\widetilde{Y}_{k-1}$ of
  generalised basic commutators of weight at most $k-1$, and that the
  inequalities in (i) hold for $H(k-1)$ in place of~$H$ and
  $1 \le n \le k-1$.  For $n \ge k$, consider
  \[
    \beta_n = \frac{\sum_{i=1}^n \dim_{\F_p} (H(k-1)_i)}{l^\circ(n)},
  \]
  where $H(k-1)_i = H(k-1) \cap \Lambda_i$ denotes the $i$th
  homogeneous component of the graded subalgebra $H(k-1)$.
  
  First suppose that $\beta_k \le \alpha$.  From
  \begin{multline*}
    \frac{\sum_{i=1}^{k-1} \dim_{\F_p} (H(k-1)_i) +
      \dim_{\F_p}(\Lambda_k^\circ)}{l^\circ(k)} \ge \left( \alpha -
      \tfrac{1}{l^\circ(k-1)} \right) \tfrac{l^\circ(k-1)}{l^\circ(k)} +
    \tfrac{l^\circ(k)- l^\circ(k-1)}{l^\circ(k)} \\ \ge \alpha-
    \tfrac{1}{l^\circ(k)},
  \end{multline*}
  we deduce that there is a finite set
  $\widetilde{Z}_k \subseteq \Lambda_k^\circ \smallsetminus H(k-1)_k$,
  consisting of generalised basic commutators of weight $k$, such that
  \[
    \alpha- \tfrac{1}{l^\circ(k)} \leq \frac{\sum_{i=1}^{k-1}
      \dim_{\F_p} \big( H(k-1)_i \big) + \dim_{\F_p} \big( H(k-1)_k
      \oplus \operatorname{\F_p-span} \widetilde{Z}_k \big)}{l^\circ(k)}\leq
    \alpha.
  \]
  Let $H(k)$ be the $\F_p[\pi]$-Lie subalgebra generated by
  $\widetilde{Y}_k = \widetilde{Y}_{k-1} \cup \widetilde{Z}_k$.
  Observe that the inequality in~(i), respectively~(ii), holds for
  $H(k)$ in place of $H$ and $1 \le n \le k$, respectively $n=k$.

  Now suppose that $\beta_k > \alpha$.  Let
  $Y_{k-1} \subseteq L^\circ$ be the finite set consisting of the core
  parts of the elements of~$\widetilde{Y}_{k-1}$, and let $M$ denote the
  $\F_p$-Lie subalgebra of $L$ generated by~$Y_{k-1}$.  By
  Proposition~\ref{prop:fg-graded-0}, the subalgebra $M$, has strong
  density $\dens_L(M) = 0$ in~$L$.  Clearly, $H(k-1)$ is contained in
  the $\F_p[\pi]$-Lie algebra generated by $Y_{k-1}$.  By
  Lemma~\ref{lem:fg-graded-0-transfer}, the latter has strong density
  $0$ in~$\Lambda$, hence so does~$H(k-1)$ in~$\Lambda^\circ$.  In
  particular, we find a minimal $k_0 \geq k+1$ such that
  $\beta_{k_0} \le \alpha$.  Putting $H(k_0-1) =H(k_0-2) = \ldots = H(k)= H(k-1)$ and
  $\widetilde{Y}_{k_0-1} =\widetilde{Y}_{k_0-2} = \ldots = \widetilde{Y}_k= \widetilde{Y}_{k-1}$, we return to the
  previous case for $k_0 > k$ in place of~$k$.
\end{proof}

Our proof of Theorem~\ref{thm:free-pro-p} uses, in addition to
Theorem~\ref{thm:all-densities-free-lie}, a partial correspondence
between selected Lie subalgebras of~$\Lambda$ and subgroups of the
free pro\nobreakdash-$p$ group~$F$, which gives rise to the mixed Lie
algebra $\boldsymbol{\Lambda}$ in the first place.  The correspondence
relies on a direct application of Philip Hall's collection process,
which we presently recall, following almost verbatim the account
provided in Section~4.3 of a preprint version
(\texttt{arXiv:1901.03101v2}) of~\cite{GaGaKl20}.

We describe the collection process in a free group $\Gamma$ on
finitely many generators $x_1, \ldots, x_d$; by taking homomorphic
images it can subsequently be applied to arbitrary $d$-generated
groups with a chosen set of $d$ generators.  The \emph{weight} of
group commutators in $x_1, \ldots, x_d$ and \emph{basic group
  commutators} in $x_1, \ldots, x_d$, including a total ordering, are
defined in complete analogy to Definition~\ref{Defi:basic_comm_Lie}.
At the heart of the Hall collection process for groups is the basic
identity
\[
uv = vu[u,v], \qquad \text{valid for all $u, v \in \Gamma$.}
\]

A finite product
\begin{equation}\label{eq:positive_word}
w = c_1 \cdots c_k c_{k+1} \cdots c_m
\end{equation}
of basic commutators $c_1, \ldots, c_m$ is \emph{ordered} if
$c_1 \le \ldots \le c_m$.  The \emph{collected part} of $w$ is the
longest prefix $c_1 \cdots c_k$ such that $c_1 \le \ldots \le c_k$ and
$c_k \le c_l$ for all $l \in \{ k+1, \ldots, m\}$.  The rest of the
product is called the \emph{uncollected part}.

The collection process produces ordered product expressions for $w$
modulo uncollected terms of higher weight, as follows.  Suppose that the leftmost occurrence of the smallest basic commutator in the uncollected part of the expression~\eqref{eq:positive_word} for~$w$ is
$c_l$, in position~$l$.  Rewriting $c_{l-1} c_l$, we obtain a new
expression for~$w$, with $m+1$ factors:
\[
w = c_1 \cdots c_k c_{k+1} \cdots
c_l c_{l-1} [c_{l-1}, c_l] c_{l+1} \cdots c_m,
\]
where the new factor $[c_{l-1}, c_l]$ is a basic commutator of higher
weight than $c_l$; in particular, $[c_{l-1}, c_l] > c_l$.  If $c_l$,
now occupying position $l-1$ in the product, is still in the
uncollected part, we repeat the above step until $c_l$ becomes the rightmost commutator in the collected part.  For any given
$r \in \N$, carrying out finitely many iterations of the above
procedure leads to an ordered product expression for~$w$, modulo
$\gamma_{r+1}(\Gamma)$, the $(r+1)$th term of the lower central series
of~$\Gamma$.

In order to apply the collection process to arbitrary elements of
$\Gamma$, i.e.\ to products in $x_1, \ldots, x_d$ and their inverses,
one would need to accommodate also for inverses $c_k^{\, -1}$ of basic
commutators as factors in our expressions.  That is, one would have to
consider collecting $u$ or $u^{-1}$ in expressions $v^{-1}u$,
$vu^{-1}$ or $v^{-1}u^{-1}$.  This can be done, as explained
in~\cite[\S 11]{Ha76}, but the group words we have to deal with are of
the simple form~\eqref{eq:positive_word}.

The collection process can be used to show that, for every element
$w \in \Gamma$ and every $r \in \N$, there is a unique
approximation
\[
w \equiv_{\gamma_{r+1}(\Gamma)} b_1^{\, j_1} \cdots b_t^{\, j_t}
\]
of $w$ as a product of powers of the basic commutators
$b_1, \ldots, b_t$ of weight at most $r$, appearing in order;
see~\cite[Thm.~11.2.4]{Ha76}.  The existence of the decomposition
follows from the collection process; the uniqueness requires further
considerations.

\begin{proof}[Proof of Theorem~\ref{thm:free-pro-p}]
  We may regard $F = \widehat{\Gamma}_p$ as the pro\nobreakdash-$p$ completion of
  the free group $\Gamma$ on $x_1, \ldots, x_d$, with $d \ge 2$.  The
  trivial subgroup $\{1\}$ and $F$ have Hausdorff dimensions $0$
  and~$1$.  Now suppose that $\alpha \in (0,1)$.  Using
  Theorem~\ref{thm:all-densities-free-lie}, we find an $\F_p[\pi]$-Lie
  subalgebra $H \le \Lambda$ that is contained in
  $\Lambda^\circ = [\Lambda,\Lambda]$, generated by an infinite set
  $\widetilde{Y}$ of generalised basic Lie commutators and satisfies
  $\dens_L(H) = \alpha$.  There is no harm in assuming that the set
  $\widetilde{Y}$ is minimal subject to generating~$H$.  In
  particular, this implies that no two elements of $\widetilde{Y}$
  have the same core part.  From
  $\widetilde{Y} \subseteq H \subseteq \Lambda^\circ$ we deduce that
  the core parts of all elements of $\widetilde{Y}$ lie in~$L^\circ$.

  There is a naive one-to-one correspondence between basic Lie
  commutators in $\mathsf{x}_1, \ldots, \mathsf{x}_d$ and basic group
  commutators in $x_1, \ldots, x_d$, which works by replacing
  occurrences of $\mathsf{x}_i$ by $x_i$ and exchanging Lie brackets
  with group commutators.  If
  $\tilde{\mathsf{c}} = \pi^k \mathsf{c} \in \Lambda$ is a generalised
  basic Lie commutator, with $k \in \N_0$ and core part $\mathsf{c}$,
  and if $c \in \Gamma$ is the basic group commutator derived from
  $\mathsf{c}$, we call $c^{p^k}$ the generalised basic group
  commutator derived from $\tilde{\mathsf{c}}$, with core part~$c$.
  Let $\widetilde{Y}_\mathrm{grp}$ denote the set of generalised basic
  group commutators derived from the elements of~$\widetilde{Y}$.
  Consider the subgroup
  $\Delta = \langle \widetilde{Y}_\mathrm{grp} \rangle \le \Gamma$;
  from $\widetilde{Y} \subseteq \Lambda^\circ$ we see that
  $\Delta \subseteq [\Gamma,\Gamma]$.  Let $\overline{\Delta}$ denote
  the topological closure of $\Delta$ in~$F$.

  The canonical embedding
  $\Gamma \hookrightarrow \widehat{\Gamma}_p = F$ induces canonical
  isomorphisms
  $P_n(\Gamma)/P_{n+1}(\Gamma) \cong P_n(F)/P_{n+1}(F) = \Lambda_n$
  for $n \in \N$.  We change perspective and regard
  $\boldsymbol{\Lambda} = (\Lambda, \mathsf{P})$ as the mixed Lie
  algebra associated to the discrete group~$\Gamma$, with respect to
  its lower $p$-series.  Let
  \[
    \varphi \colon \Gamma \to \Lambda = \bigoplus_{n=1}^\infty
    \Lambda_n, \quad \text{with components
      $\Lambda_n = P_n(\Gamma)/P_{n+1}(\Gamma)$,}
  \]
  denote the canonical map that sends $1$ to $1 \varphi = 0$ and each
  non-identity element
  $w \in P_n(\Gamma) \smallsetminus P_{n+1}(\Gamma)$ to
  $w \varphi = w P_{n+1}(\Gamma) \in \Lambda_n$, for $n \in \N$.  It
  is easily seen that
  $\Delta \varphi \subseteq [\Gamma,\Gamma] \varphi = \Lambda^\circ$
  is an $\F_p[\pi]$-Lie subalgebra and that
  \[
    \hdim_F^\mathcal{L}(\overline{\Delta}) = \dens_{\Lambda} (\Delta
    \varphi).
  \]
  Hence it suffices to show that $\Delta \varphi = H$; this duly
  implies $\alpha \in \hspec^{\mathcal{L}}(F)$.

  We observe that, for every basic group
  commutator~$c \in [\Gamma,\Gamma]$ and for $k \in \N_0$, the Lie
  element $(c^{p^k}) \varphi$ is equal to the generalised basic Lie
  commutator $\pi^k \mathsf{c} \in \Lambda_n$, where
  $\mathsf{c} = c^\mathrm{Lie}$ is the basic Lie commutator
  corresponding to~$c$ and $n = k+\wt(\mathsf{c})$.  In particular,
  the construction ensures that
  $\widetilde{Y}_\mathrm{grp} \varphi = \widetilde{Y}$.  Thus
  $\Delta \varphi$ contains $\widetilde{Y}$, and
  $H \subseteq \Delta \varphi$.  It remains to show that, conversely,
  $w \varphi \in H$ for every non-trivial element $w \in \Delta$.

  Let $w \in \Delta \smallsetminus \{1\}$, and choose $r \in \N$ such
  that $w \not \in P_{r+1}(\Gamma)$.  The element $w$ is a finite
  product of generalised basic group commutators from the generating
  set $\widetilde{Y}_\mathrm{grp}$ of~$\Delta$ and their inverses.
  However, working modulo the finite-index subgroup $P_{r+1}(\Gamma)$,
  we can avoid using inverses and assume that
  $w = \tilde c_1 \cdots \tilde c_m$ with $m \in \N$ and
  $\tilde c_1, \ldots, \tilde c_m \in \widetilde{Y}_\mathrm{grp}$, similar
  to~\eqref{eq:positive_word}.

  We now apply the collection process, treating the generalised basic
  commutators $\tilde c_i = c_i^{\, p^{k(i)}}$ as if they were equal
  to their core parts $c_i$ carrying the exponents $p^{k(i)}$ simply
  as decorations; we recall that, by the minimal choice of
  $\widetilde{Y}$, each exponent $p^{k(i)}$ is uniquely determined by
  the core part~$c_i$.  As
  $P_{r+1}(\Gamma) \supseteq \gamma_{r+1}(\Gamma)$, the collection
  process shows that, modulo $P_{r+1}(\Gamma)$, there is a finite
  product decomposition
  \begin{equation}\label{eqn:collected_w}
    w \equiv_{P_{r+1}(\Gamma)} \tilde{b}_1^{\, p^{e(1)} j_1} \, \tilde{b}_2^{\, p^{e(2)} j_2}
    \, \cdots \, \tilde{b}_t^{\, p^{e(t)} j_t},
  \end{equation}
  where (i) $\tilde{b}_1, \ldots, \tilde{b}_t$ are commutator
  expressions in $\tilde c_1, \ldots, \tilde c_m$ which upon replacing
  $\tilde c_1, \ldots, \tilde c_m$ by their core parts
  $c_1, \ldots, c_m$ yield basic group commutators $b_1, \ldots, b_t$
  that appear in order, and (ii)
  $e(1), \ldots, e(t), j_1, \ldots, j_t \in \N_0$ with
  $p \nmid j_1 \cdots j_t$.

  For $1 \le s \le t$, we define the weight of the expression
  $\tilde{b}_s$ to be
  \[
    \wt(\tilde{b}_s) = \wt(b_s) + A_s, \qquad \text{with} \quad A_s =
    \sum_{i=1}^m a_{s,i} \, k(i),
  \]
  where $\wt(b_s)$ is the weight of the corresponding `core'
  commutator and $a_{s,i}$ denotes the number of occurrences of $c_i$
  in~$b_s$.  Let $\mathsf{b}_s = b_s^\mathrm{Lie}$ denote the basic
  Lie commutator corresponding to~$b_s$.  Induction on $\wt(b_s)$,
  which measures the `length' of the commutator expression
  $\tilde{b}_s$, shows that
  \begin{equation} \label{equ:bs-phi}
    \tilde{b}_s \in P_{\wt(\tilde{b}_s)}(\Gamma) \smallsetminus
    P_{\wt(\tilde{b}_s) + 1}(\Gamma) \qquad \text{and} \qquad
    \tilde{b}_s \varphi = \pi^{A_s} \mathsf{b}_s.
  \end{equation}
  
  Let
  \[
    n = \min \big\{ e(s) + \wt(\tilde{b}_s) \mid 1 \le s \le t \big\}
  \]
  and set
  \[
    S = \big\{ s \mid 1 \le s \le t \text{ and } e(s) + \wt(\tilde{b}_s) = n
    \big\}.
  \]
  Then $w \in P_n(\Gamma) \smallsetminus P_{n+1}(\Gamma)$ and
  $w \varphi \in \Lambda_n$ is obtained from~\eqref{eqn:collected_w}
  by ignoring those factors $\tilde{b}_s^{\, p^{e(s)} j_s}$ for which
  $e(s) + \wt(\tilde{b}_s) > n$.  Indeed, \eqref{eqn:collected_w} and
  \eqref{equ:bs-phi} show that $w \varphi$ can be written as a
  non-trivial linear combination
  \[
    w \varphi = \sum_{s \in S} j_s \, \pi^{e(s)} \tilde{\mathsf{b}}_s,
    \qquad \text{with $\tilde{\mathsf{b}}_s = \pi^{A_s} \mathsf{b}_s$
      for $s \in S$},
  \]
  of the $\F_p$-linearly independent generalised basic Lie commutators
  $\pi^{e(s) + A_s} \mathsf{b}_s = \pi^{e(s)} \tilde{\mathsf{b}}_s$,
  $s\in S$, of weight~$n$.

  Each $\tilde{\mathsf{b}}_s$ is a Lie commutator of the generalised
  basic Lie commutators
  $\tilde{\mathsf{c}}_1 = \pi^{k(1)} \mathsf{c}_1, \ldots,
  \tilde{\mathsf{c}}_m = \pi^{k(m)} \mathsf{c}_m \in \widetilde{Y}$,
  in the same way as $\tilde{b}_s$ is a group commutator of
  $\tilde{c}_1, \ldots, \tilde{c}_m$.  As observed before, we have
  $\tilde{\mathsf{c}}_i = \tilde{c}_i \varphi \in \widetilde{Y}$ for
  $1 \le i \le m$.  Thus, $w \varphi$ lies indeed in~$H$.

  \smallskip

  Up to this point we assumed, for convenience, that $p > 2$.  This
  meant that we could regard the mixed Lie algebra
  $\boldsymbol{\Lambda}$ as a free $\F_p[\pi]$-Lie algebra which
  simplified our discussion.  As indicated earlier, for $p=2$, the
  mixed Lie algebra $\boldsymbol{\Lambda}$ has a slightly more
  complicated structure; but the mixed subalgebra
  $\boldsymbol{\Lambda}^\circ$ based on
  $\Lambda^\circ = [\Lambda,\Lambda]$ retains its properties.
  Moreover, it remains to settle the case that
  $\dim_{\F_p}(\Lambda_n) = d + \dim_{\F_p}(\Lambda_n^\circ)$ for each
  $n \in \N$.  By going through our proofs up to this point, it is
  easily seen that these ingredients are all we really need.  Thus we
  arrive at the same conclusions, also for $p=2$, notwithstanding the
  fact that $\boldsymbol{\Lambda}$ has a somewhat different structure.
\end{proof}

It remains to generalise our result to finite direct products of free
pro\nobreakdash-$p$ groups.  We use, without any substantial change, the line of
reasoning laid out in the proof of Theorem~1.3 in a preprint version
(\texttt{arXiv:1901.03101v2}) of~\cite{GaGaKl20}.

\begin{proof}[Proof of Theorem~\ref{thm:direct-prod-of-free-pro-p}]
  We consider the finite direct product
  $G = F_1 \times \ldots \times F_r$ of finitely generated free
  pro\nobreakdash-$p$ groups~$F_j$, where
  $d = \max \{d(F_j) \mid 1 \le j \le r \} \ge 2$.  Observe that the
  lower $p$-series $\mathcal{L} \colon P_n(G)$, $n \in \N$, of $G$
  decomposes as the product of the lower $p$-series
  $\mathcal{L}_j \colon P_n(F_j)$, $n \in \N$, of the direct
  factors~$F_j$, for $1 \le j \le r$.  Setting
  $t = \# \{ j \mid 1 \le j \le r, \,d(F_j) = d \}$, we may assume
  that $d(F_j) = d$ for $1 \le j \le t$ and $d(F_j) < d$ for
  $t+1 \le j \le r$.  Lemma~\ref{lem:Witt-for-mixed} shows that, for
  $t+1 \le j \le r$,
  \begin{equation} \label{equ:small-factors-irrelevant} \lim_{n \to
      \infty} \frac{\log_p \lvert F_j : P_n(F_j) \rvert}{\log_p \lvert
      F_1 : P_n(F_1) \rvert} = 0.
  \end{equation}

  Now let $\alpha \in [0,1]$ and choose $k \in \{1,\ldots,t\}$ such
  that $\nicefrac{(k-1)}{t} \le \alpha \le \nicefrac{k}{t}$.
  By Theorem~\ref{thm:free-pro-p}, we find a subgroup
  $H_1 \le_c F_1$ with
  $\hdim_{F_1}^{\mathcal{L}_1}(H_1) = t \alpha - (k-1)$.  Then
  \[
  H = H_1 \times F_2 \times \ldots \times F_k \times 1 \times \ldots
  \times 1 \le_\mathrm{c} G
  \]
  has Hausdorff dimension
  \begin{align}
    \hdim_G^{\mathcal{L}}(H) 
    & = \varliminf_{n \to \infty} \frac{\log_p\lvert H_1 P_n(F_1)
      : P_n(F_1) \rvert + \sum_{j=2}^k \log_p \lvert F_j :
      P_n(F_j) \rvert + 0}{\sum_{j=1}^t \log_p \lvert F_j :
      P_n(F_j) \rvert +
      \sum_{j=t+1}^r \log_p \lvert F_j : P_n(F_j) \rvert}
      \nonumber \\
    & = \varliminf_{n \to \infty} \frac{\log_p\lvert H_1 P_n(F_1)
      : P_n(F_1) \rvert + (k-1) \log_p \lvert
      F_1 : P_n(F_1) \rvert}{t \log_p \lvert F_1 : P_n(F_1) \rvert
      }  \label{equ:prod-calculation} \\
    & = \frac{\hdim_{F_1}^{\mathcal{L}_1}(H_1) + (k-1)}{t}
      \nonumber \\
    & = \alpha. \nonumber \qedhere
  \end{align}
\end{proof}


\section{Just infinite or nilpotent $p$-adic analytic  pro-\texorpdfstring{$p$}{p} groups} \label{sec:just-infinite-or-nilpotent}

We first recall the set-up and notation from \cite{HeKlTh25}, which  considers a series, analogous to the lower $p$-series of $p$-adic analytic pro\nobreakdash-$p$ groups,  for $\Z_p$-lattices furnished with a group action. 
 
 Let $G$ be a finitely generated pro-$p$ group and let $L$ be
a $\Z_p$-lattice equipped with a (continuous) right $G$-action. We denote the kernel of the natural epimorphism from $\Z_pG$ onto the finite field $\F_p$, that sends each group element to~$1$, by 
\[
  \aug_G = \sum_{g\in G} (g-1) \Z_pG +
  p\Z_pG \,\trianglelefteq\, \Z_pG.
\]
The descending series of open $\Z_pG$-submodules
\begin{equation*}
\lambda_i(L) = L . \aug_G^{\, i} \quad \text{for $i \in \N_0$,}
\end{equation*}
is called the \emph{lower $p$-series} of the
$\Z_pG$-module $L$; note however the notational shift in the index
in comparison to the lower $p$\nobreakdash-series of a group.  As $G$
is a pro\nobreakdash-$p$ group, it acts unipotently on every principal
congruence quotient $L/p^j L$, $j \in \N$, and this gives
that $\bigcap_{i \in \N_0} \lambda_i(L) = \{0\}$.

A \emph{filtration series} of $L$, regarded as a
$\Z_p$-lattice, is defined as a descending chain
$\mathcal{S} \colon L = L_0\ge L_1 \ge \ldots$ of open
$\Z_p$-sublattices $L_i \le_\mathrm{o} L$ such that
$\bigcap_{i \in \N_0} L_i = \{0\}$.  For $c \in \N_0$,
we say that two filtration series
$\mathcal{S} \colon L = L_0 \ge L_1 \ge \ldots$ and
$\mathcal{S}^* \colon L = L_0^* \ge L_1^* \ge \ldots$ are
\emph{$c$-equivalent} if for all $i \in \N$,
\[
  p^c L_i \subseteq L_i^* \quad \text{and} \quad p^c L_i^* \subseteq
  L_i.
\]
We say that $\mathcal{S}$ and $\mathcal{S}^*$ are \emph{equivalent} if
they are $c$-equivalent for some $c \in \N_0$.

\smallskip

For convenience, we state the following  result from \cite{HeKlTh25}, which is a variation of
\cite[Prop.~4.3]{Kl03}. First we recall that the \emph{rigidity} of~$L$ is
\[
  r(L) = \sup \{\ell_L(M) - u_L(M) \mid \text{$M$ is an open
    $\Z_pG$-submodule of $L$} \},
\]
where for an open $\Z_pG$-submodule
$M$ in~$L$, 
\[
  \ell_L(M) = \min \{k\in \N_0 \mid p^kL \subseteq M \} \;\text{and}\;
  u_L(M) = \max \{k\in \N_0 \mid M \subseteq p^kL\};
\]
cf.~\cite{Kl03}. For $L = \{0\}$ we set $r(L)= -\infty$.

\begin{proposition}
  \label{pro:rigid}\cite[Prop.~2.2]{HeKlTh25}
  Let $G$ be a finitely generated pro\nobreakdash-$p$ group, and let $L$ be a
  $\Z_p$-lattice equipped with a right $G$-action.  Then
  $r(L)$ is finite if and only if
  $\Q_p\otimes_{\Z_p} L$ is a simple
  $\Q_pG$-module.
\end{proposition}

\smallskip

Now we turn to the special situation where the $p$-adic analytic group is just infinite.  This class of groups is of considerable interest and has been studied in some detail; for instance, see \cite{KLP97}.  We start with two somewhat more general auxiliary results. 

\begin{proposition}\label{pro:simple-L-P}
  Let $G$ be an infinite $p$-adic analytic pro\nobreakdash-$p$ group,
  and let $\mathcal{S}$ be a filtration series of~$G$.  Let $L$ be the
  $\Z_p$-Lie lattice associated to a uniformly powerful open normal
  subgroup $U \trianglelefteq_\mathrm{o} G$, equipped with the induced
  $G$-action, and suppose that $\Q_p\otimes_{\Z_p} L$ is a simple
  $\Q_pG$-module.
  
  Then every $H \le_\mathrm{c} G$ has strong Hausdorff dimension
 \[
   \hdim_G^{\mathcal{S}}(H) = \frac{\dim(H)}{\dim(G)}.
  \]
\end{proposition}

\begin{proof}
  Let $H \le_\mathrm{c} G$.  Without loss of generality, we may suppose
  that every open normal subgroup $N \trianglelefteq_\mathrm{o} G$
  with $N \subseteq U$ is powerfully embedded in~$U$; see
  \cite[Prop.~3.9]{DDMS99}.  Then
  $\mathcal{S} \vert_U \colon U_i = G_i \cap U, \, i \in \N_0$,
  is a filtration series of $U$, consisting of uniformly powerful
  subgroups, and
  \begin{equation}\label{equ:hdim-in-G=hdim-in-U}
    \hdim^{\mathcal{S}}_G(H) = \hdim^{\mathcal{S}}_G(H\cap U) =
    \hdim^{\mathcal{S}\vert_U}_U(H\cap U);
  \end{equation}
  compare with~\cite[Lem.~2.1]{KlTh19}.  

  The series $\mathcal{S} \vert_U$ translates over to the Lie lattice side to a filtration series
  $\mathcal{S} \vert_L \colon L_i, \, i \in \N_0$, of $L$.     By Proposition~\ref{pro:rigid} the
  $\Z_pG$-module $L$ has finite rigidity.  Consequently,
  $\mathcal{S} \vert_L$ is equivalent to a filtration series of the
  form
  $\mathcal{S} \vert_L^* \colon L \ge p^{n_1}L \ge p^{n_2}L \ge
  \ldots$, for a suitable non-decreasing sequence of non-negative
  integers $n_1, n_2, \ldots$ tending to infinity.

  Pick a uniformly powerful open subgroup $W \le_\mathrm{o} H \cap U$,
  and let $M \le L$ denote the corresponding Lie sublattice.   We deduce that
  \begin{equation} \label{equ:hdim-in-U=hdim-in-L}
    \hdim^{\mathcal{S}\vert_U}_U(H\cap U) =
    \hdim^{\mathcal{S}\vert_U}_U(W) = \hdim^{\mathcal{S}\vert_L}_L(M)
    = \hdim^{\mathcal{S}\vert_L^*}_L(M);
  \end{equation}
 indeed, in the first equality we use \cite[Lem.~2.1]{KlTh19}, in the second equality we are using Equation (3.3) of the proof of \cite[Thm.~1.5]{HeKlTh25}, while in the third equality we use \cite[Lem.~2.2]{KlThZu19}.
  A direct inspection shows that the additive subgroup $M$ of $L$ has
  strong Hausdorff dimension
  \[
    \hdim^{\mathcal{S}\vert_L^*}_L(M) = \frac{\dim(M)}{\dim(L)} =
    \frac{\dim(H)}{\dim(G)}.
  \]
  Tracing our way back through \eqref{equ:hdim-in-G=hdim-in-U} and
  \eqref{equ:hdim-in-U=hdim-in-L} concludes the proof.
\end{proof}

For the next result, we recall from ~\cite[Sec.~4 and~9]{DDMS99} that  there is an explicit isomorphism of categories translating between
  uniformly powerful pro\nobreakdash-$p$ groups and powerful
  $\Z_p$-Lie lattices.  So for a  uniformly powerful pro\nobreakdash-$p$ group $U$,  the underlying set of  $U$ can be equipped in a canonical way
  with the structure of a $\Z_p$-Lie lattice $L$ carrying the same
  topology.  In particular,  exponentiation in~$U$  corresponds to scalar multiplication in $L$ and  the conjugation
  action of~$G$ on $U$ translates to a continuous $\Z_p$-linear action
  of $G$ on~$L$.  
\begin{lemma} \label{lem:a-b-c-equiv} Let $U$ be a uniformly powerful
  pro\nobreakdash-$p$ group, and let $L$ be the associated $\Z_p$-Lie
  lattice.  Then the following are equivalent:
  \begin{enumerate}
  \item[(a)] $U$ is insoluble and just infinite.
  \item[(b)] $\Q_p \otimes_{\Z_p} L$ is a simple $\Q_p$-Lie
    algebra.
  \item[(c)] $\Q_p \otimes_{\Z_p} L$ is a simple 
    $\Q_p U$-module with non-trivial $U$-action.
  \end{enumerate}
\end{lemma}

\begin{proof}
  It is known that (a) implies (b); for instance, see
  \cite[Prop.~III.6]{KLP97} and its proof.

  \smallskip
  
  Next we show that (b) implies (c).  Suppose that
  $\mathfrak{L} = \Q_p \otimes_{\Z_p} L$ is a simple
  $\Q_p$-Lie algebra.  Let $\mathfrak{M}$ denote a non-zero
  $\Q_p U$-submodule of $\mathfrak{L}$.  In order to show that
  $\mathfrak{M} = \mathfrak{L}$ it suffices to prove that
  $M = \mathfrak{M} \cap L$ is a Lie ideal in~$L$.
  
For clarity, we write $\underline{x}$ for
  $x \in U$ when it features as an element of the Lie lattice~$L$
  rather than as a group element.  Let $\underline{x} \in M$ and $\underline{y} \in L$. Since
  $M \subseteq_\mathrm{c} L$, it suffices to prove that
  \begin{equation} \label{equ:comm-contained-mod}
    [\underline{x},\underline{y}]_{\mathrm{Lie}} \in M + p^{k-1} L
    \qquad \text{for all $k \in \N$ with
      $k \ge 2$.}
  \end{equation}
  Let $k \in \N$ with
  $k \ge 2$.  From \cite[IV (3.2.7)]{La65} we see that
  \[
    \underline{x}.y^{p^k} - \underline{x} = \sum_{n=1}^\infty
    \tfrac{1}{n!} \, [\underline{x},_n p^k \underline{y}]_\mathrm{Lie} =
    p^ k \, [\underline{x}, \underline{y}]_\mathrm{Lie} +
    \sum_{n=2}^\infty \tfrac{p^{nk}}{n!} \, [\underline{x},_n
    \underline{y}]_\mathrm{Lie}.
  \]
  We recall that
  $(n!)^{-1} \in p^{-\lfloor (n-1)/(p-1) \rfloor} \Z_p$ and
  observe that for $n \ge 2$,
  \[
    nk - \lfloor (n-1)/(p-1) \rfloor \ge k + (k-1)(n-1) \ge k + (k-1).
  \]
  This gives
  \[
    [\underline{x}, \underline{y}]_\mathrm{Lie} \equiv_{p^{k-1}L}
    p^{-k} \big( \underline{x}.y^{p^k} - \underline{x} \big) \in \mathfrak{M},
  \]
  and \eqref{equ:comm-contained-mod} holds.

  \smallskip
  
  Finally, we argue that (c) implies (a).  Suppose that
  $\mathfrak{L} = \Q_p \otimes_{\Z_p} L$ is a simple 
    $\Q_p U$-module with non-trivial $U$-action.  Let $N \trianglelefteq_\mathrm{c} U$ be a
  non-trivial normal subgroup.  Then $N$ contains a uniformly powerful
  characteristic open subgroup~$K$, which supplies a corresponding
  non-trivial Lie sublattice $M \le_\mathrm{c} L$.  Since $K$ is
  normal in~$U$, the lattice $M$ is $U$-invariant.  This implies that
  the $\Q_p U$-submodule $\Q_p \otimes_{\Z_p} M$
  equals $\mathfrak{L}$ and gives
  \[
    \dim(N) = \dim(K) = \dim_{\Z_p}(M) =
    \dim_{\Z_p}(L) = \dim(U).
  \]
  Hence $N$ is open in~$U$. This proves that $U$ is just infinite and also implies that $U$ is insoluble, for otherwise the uniform group $U$ would be abelian (in fact procyclic) and the action of $U$ on $\Q_p \otimes_{\Z_p} L$ would be trivial, contrary to~(c).
\end{proof}

Now we prove Theorem~\ref{thm:just-infinite-or-nilpotent}.

\begin{proof}[Proof of Theorem~\ref{thm:just-infinite-or-nilpotent}]
  We treat the two situations, $G$ is just infinite and $G$ is
  nilpotent, one after the other.

  \smallskip
  
  \noindent (1) Suppose that $G$ is a just infinite $p$-adic analytic
  pro\nobreakdash-$p$ group.

  \smallskip
  
  \noindent\textsl{Case 1:} $G$ is soluble.
  In this situation $G$ is virtually abelian and, in fact, an
  irreducible $p$-adic space group: there is an abelian open normal
  subgroup $A \trianglelefteq_\mathrm{o} G$ such that
  $A \cong \Z_p^d$ is a $\Z_p$-lattice and the finite
  $p$-group $\Gamma =G/A$ acts faithfully on~$A$ and irreducibly
  on~$\Q_p \otimes_{\Z_p} A$.  Thus the claims follow
  from Proposition~\ref{pro:simple-L-P}.
  
  \smallskip
  
  \noindent\textsl{Case 2:} $G$ is insoluble.
  Let $U$ be a uniformly powerful open normal subgroup of~$G$, and let
  $L$ be the Lie lattice associated to~$U$.  The $\Q_p$-Lie
  algebra $\mathfrak{L} = \Q_p\otimes_{\Z_p} L$ is
  known to be semisimple of homogeneous type:
  $\mathfrak{L} = \bigoplus_{k=1}^q \mathfrak{L}_k$, for some
  $p$-power~$q$, with simple components
  $\mathfrak{L}_1 \cong \ldots \cong \mathfrak{L}_q$;
  see~\cite[Prop.~III.6]{KLP97}.  Furthermore, the proof
  of~\cite[Prop.~III.6]{KLP97} shows that~$G$ permutes transitively
  the $q$ components~$\mathfrak{L}_k$ of~$\mathfrak{L}$.  Using
  Lemma~\ref{lem:a-b-c-equiv}, we deduce that $\mathfrak{L}$ is a
  simple $\Q_p G$-module and the result follows from
  Proposition~\ref{pro:simple-L-P}.

  \medskip

  \noindent (2) Suppose that the $p$-adic analytic pro\nobreakdash-$p$
  group $G$ is nilpotent of class~$c$.    Referring to \cite[Prop.~3.9 and Thm.~4.2]{DDMS99},  let $j \in \N$ be such
  that $U = P_j(G)$ is uniformly powerful and $P_i(G)$ is powerfully
  embedded in $U$ for all $i \in \N$ with $i \ge j$.  Let $L$ be the
  $\Z_p$-Lie lattice associated to~$U$, equipped with the induced
  $G$-action, and let $L_i = \lambda_{i}(L)$ 
  denote the Lie sublattices corresponding to
  $P_{i+j}(G)$, which form with the lower $p$-series
  of~$L$.  

  Intersecting the terms $\gamma_s(G)$ of the lower central series of
  $G$ with $U$, we arrive at a descending series
  \[
    U = V_1 \ge V_2 \ge \ldots \ge V_{c+1} = 1
  \]
  of closed normal subgroups of $G$ such that
  $[V_s,G] \subseteq V_{s+1}$ for $1 \le s \le c$.  For each $s \in
  \{1, \ldots, c+1\}$, let
  \[
    W_s = \{ x \in U \mid \exists m \in \N : x^{p^m} \in V_s
    \} \le_\mathrm{c} U
  \]
  denote the isolator of $V_s$ in $U$, a powerful subgroup that is
  normal in~$G$; compare with \cite[Scholium to Thm.~9.10]{DDMS99}.
  Raising elements to some power commutes with applying
    the $G$-action; hence it is straightforward to check that
  $[W_s,G] \subseteq W_{s+1}$ for $1 \le s \le c$.  On the Lie lattice
  side, we arrive at a corresponding descending series
  \[
    L = M_1 \ge M_2 \ge \ldots \ge M_{c+1} = \{0\}
  \]
  of Lie sublattices such that $G$ acts trivially on each section
  $M_s/M_{s+1}$, $1 \le s \le c$.  Writing
  $\mathfrak{b} = \sum_{g\in G} (g-1) \Z_pG$ we deduce that,
  for $i \in \N$, 
  \[
   L_i = \lambda_{i}(L) = L.\aug_G^{\, i} 
   = L. \Big(
    \mathfrak{b} + p\Z_p \Big)^i \subseteq \sum_{s=0}^c
    p^{i-s} L.\mathfrak{b}^s
  \]
  and hence
  \[
   p^{i} L \subseteq L_{i} \subseteq p^{i-c} L,
  \]
  where negative powers of $p$ are interpreted by embedding $L$ into
  $\Q_p \otimes_{\Z_p} L$.  In other words, the filtration
  series $L_{i}$, $i \in \N$,  
  and
  $p^{i} L$, $i \in \N$,
  are $c$-equivalent.
  Arguing as in the proof of Proposition~\ref{pro:simple-L-P}, we
  obtain that every $H \le_\mathrm{c} G$ has strong Hausdorff
  dimension
 \[
   \hdim_G^{\mathcal{L}}(H) = \frac{\dim(H)}{\dim(G)}. \qedhere
 \]
\end{proof}


\smallskip 

\section{Large finite Hausdorff spectra}\label{sec:unbounded}

In this final section we establish
  Theorem~\ref{thm:uniform-bound}, yielding arbitrarily large finite
  Hausdorff spectra for one and the same pro\nobreakdash-$p$ group,
  when we are free to vary the underlying filtration series.

\smallskip 

By \cite[Prop.~2.1]{KlThZu19},
Theorem~\ref{thm:uniform-bound} is a consequence of the next result.

\begin{proposition}
  \label{pro:unbounded-spectrum}
  Let $L\cong\Z_p\oplus\Z_p$, and let $X \subseteq [0,1]$ be finite
  with $\{0,1\} \subseteq X$. Then there exists a filtration series
  $\mathcal{S}$ such that $\hspec^{\mathcal{S}}(L) = X$.
\end{proposition}

\begin{proof}
  We write $L = \Z_p x \oplus \Z_p y$ and
  $ X = \{\xi_1, \ldots, \xi_n\}$, where $n = \lvert X \rvert \ge 2$
  and $0 = \xi_1 < \ldots <\xi_n=1$ are real numbers.  For
  $k \in \{1,\ldots,n\}$ and $j \in \N_0$ we set
  \[
    t_{k,j} = \big\lfloor \tfrac{1}{k} \big( 2^{2^{k+jn}} -
    2^{2^{k+jn-1}} \big)(1 - \xi_k) \big\rfloor k \in k \N;
  \]
  thus $t_{k,j}$ is the maximal positive multiple of $k$ subject to
  the condition
  $t_{k,j} \le \big( 2^{2^{k+jn}} - 2^{2^{k+jn-1}} \big)(1 - \xi_k)$.
  We put $L_0 = L$, and for $k \in \{1,\ldots,n\}$ and $j \in\N_0$, we
  set
 \begin{equation*}
   L_{k+jn} = p^{2^{2^{k+jn-1}}} \Z_p \tilde x_{k,j} \oplus
   p^{2^{2^{k+jn}}} \Z_p y, \qquad \text{where} \quad \tilde x_{k,j} = x +
   \frac{1-p^{k+t_{k,j}}}{1-p^k} y.
 \end{equation*}
 We observe that
 \[
   p^{2^{2^{k+jn}}} L \subseteq L_{k+jn} \subseteq p^{2^{2^{k+jn-1}}}L \quad
   \text{and} \quad \log_p \lvert L : L_{k+jn} \rvert =
   2^{2^{k+jn-1}} + 2^{2^{k+jn}}.
 \]
 In particular, $\mathcal{S} \colon L =L_0 \ge L_1 \ge \ldots$
 forms a filtration series of~$L$.

 We decompose $\N$ into its residue classes modulo~$n$,
 \[
   \N = I_1 \cup \ldots \cup I_n, \qquad \text{where} \;\; I_k = k +
   n\N_0 \;\; \text{for} \;\; 1 \le k \le n,
 \]
 and set
 \[
   z_k = x + (1-p^k)^{-1}\, y, \qquad \text{for $1 \le k \le n$}.
 \]
 We observe that each of the $1$-dimensional subgroups $\Z_p z_k$,
 $1 \le k \le n$, is isolated in~$L$, i.e.\ equal to its
   isolator in $L$, and that any two of them intersect trivially:
 $\Z_p z_k \cap \Z_p z_l = \{0\}$ for $1 \le k <l \le n$.

 The Hausdorff spectrum of $L$ with respect to $\mathcal{S}$ trivially
 contains $0$ and~$1$.  Any further points in the spectrum arise from
 $1$-dimensional subgroups.  Let $H \le L$ be a $1$-dimensional
 subgroup.  We show below that, for each $k \in \{1, \ldots, n\}$,
 \begin{equation}
   \label{eq:hdim-incl-non-incl}
   \lim_{i \in I_k} \frac{\log_p \lvert H + L_i : L_i \rvert}{\log_p
     \lvert L : L_i \rvert} =
   \begin{cases}
     \xi_k & \text{if $H \le \Z_p z_k$,} \\
     1 & \text{otherwise.}
   \end{cases}   
 \end{equation}
 This implies that
   \begin{multline*}
     \hdim_L^{\mathcal{S}}(H) = \varliminf_{i\in\N} \frac{\log_p
       \lvert H + L_i:L_i \rvert}{\log_p \lvert L : L_i \rvert} =
     \min_{1 \le l \le n} \; \lim_{i \in I_l} \frac{\log_p \lvert H +
       L_i : L_i \rvert}{\log_p
       \lvert L : L_i \rvert} \\
     =
     \begin{cases}
       \xi_k & \text{if $H \le \Z_p z_k$, for a suitable $k \in \{1,
         \ldots, n\}$,} \\
       1 & \text{otherwise,}
     \end{cases}
   \end{multline*}
 which concludes the proof.

 \smallskip

 It remains to establish \eqref{eq:hdim-incl-non-incl}.  Let
 $k \in \{1,\ldots,n\}$, and first suppose that $H \le \Z_p z_k$.
 Then $\hdim_L^\mathcal{S}(H) = \hdim_L^\mathcal{S}(\Z_p z_k)$, and it
 suffices to show that $\hdim_L^\mathcal{S}(\Z_p z_k) = \xi_k$.  From
 the definition of $t_{k,j}$ we see that, for $j \in \N$,
 \begin{equation} \label{equ:mod-L-kjn-equiv}
   p^{2^{2^{j+kn-1}}} z_k
   \equiv_{L_{k+jn}} p^{2^{2^{j+kn-1}}} (z_k - \tilde x_{k,j} ) =
   p^{2^{2^{j+kn-1}}} p^{k+t_{k,j}} (1-p^k)^{-1} \, y.
 \end{equation}
 This yields
 \begin{equation*}
   \begin{split}
     \lim_{i \in I_k} \frac{\log_p \lvert \Z_p z_k + L_i:L_i
       \rvert}{\log_p \lvert L : L_i \rvert} %
     & = \lim_{j \to \infty}\frac{\log_p \lvert \Z_p z_k + L_{k+jn} :
       L_{k+jn} \rvert}{\log_p \lvert L : L_{k+jn} \rvert}  \\
     & = \lim_{j \to \infty} \frac{2^{2^{k+jn-1}} + \big(
       2^{2^{k+jn}}-2^{2^{k+jn-1}} - k - t_{k,j} \big)}{2^{2^{k+jn-1}}
       + 2^{2^{k+jn}}} \\
     & = \lim_{j \to \infty} \frac{2^{2^{k+jn-1}} +\big(
       2^{2^{k+jn}} - 2^{2^{k+jn-1}} \big) \big(1
       -(1-\xi_k)\big)}{2^{2^{k+jn-1}} + 2^{2^{k+jn}}}\\
     & = \xi_k.
   \end{split}
 \end{equation*}

 Now suppose that $H \not \le \Z_p z_k$. We write $H = \Z_p w$, where
 $w = p^m x + b y$ with $m\in\N_0$ and $b \ne p^m (1-p^k)^{-1}$.  Let
 $l \in\N_0$ such that
 $a = b - p^m (1-p^k)^{-1} \in p^l \Z_p \smallsetminus p^{l+1}\Z_p$.
 We observe from \eqref{equ:mod-L-kjn-equiv} that for every
 $j \in \N_0$ with $k+t_{k,j} >l$, we have
 \[
   p^{2^{2^{k+jn-1}}} z_k \equiv_{L_{k+jn}} p^{2^{2^{k+jn-1}}} a_j' \,
   y \quad \text{with} \quad a_j' = p^{k+t_{k,j}} (1-p^k)^{-1} \in
   p^{l+1}\Z_p
 \]
 and hence
 \[
   p^{2^{2^{k+jn-1}} -m} w = p^{2^{2^{k+jn-1}}} p^{-m} (p^m z_k + a y)
   \equiv_{L_{k+jn}} p^{2^{2^{k+jn-1}}} (a_j' + p^{-m}a) y,
 \]
 where we embed $L$ into $\Q_p \otimes_{\Z_p} L$ to interpret negative
 powers of~$p$.  This yields
 \begin{align*}
 &  \lim_{i\in I_k} \frac{\log_p \lvert H + L_i: L_i \rvert}{\log_p
     \lvert L : L_i \rvert} = \lim_{j \to \infty} \frac{\log_p \lvert
     \Z_p w + L_{k+jn} :
     L_{k+jn} \rvert}{\log_p \lvert L : L_{k+jn} \rvert} \\
 &\qquad\quad\qquad  =\lim_{j \to \infty} \frac{\big( 2^{2^{k+jn-1}}-m \big) + \big(
     2^{2^{k+jn}} - 2^{2^{k+jn-1}} + m - l \big)}{2^{2^{k+jn-1}} +
     2^{2^{k+jn}}} = 1. \qedhere
 \end{align*}
\end{proof}



\begin{thebibliography}{99}
%
\bibitem{Ab94} A.\,G.~Abercrombie, Subgroups and subrings of profinite
  rings, \textit{Math.\ Proc.\ Camb.\ Phil. Soc.} \textbf{116} (1994),
  209--222.
%

%
\bibitem{Ba87} Yu.~A.~Bahturin, \textit{Identical relations in Lie
    algebras}, VNU Science Press, b.v., Utrecht, 1987.
%
\bibitem{BaOl15} Yu.~A.~Bahturin and A.~Olshanskii, Growth of
  subalgebras and subideals in free Lie algebras, \textit{J.\ Algebra}
  \textbf{422} (2015), 277--305.
%
\bibitem{BaSh97} Y.~Barnea and A.~Shalev, Hausdorff dimension,
  pro\nobreakdash-$p$ groups, and Kac-Moody algebras, \textit{Trans.\
    Amer.\ Math.\ Soc.}  \textbf{349} (1997), 5073--5091.
%

  %
\bibitem{DDMS99} J.\,D.~Dixon, M.\,P.\,F.~du~Sautoy, A.~Mann, and
  D.~Segal, \textit{Analytic pro\nobreakdash-$p$ groups}, Cambridge
  University Press, Cambridge, 1999.
  
  
\bibitem{FA} J.~Fari\~{n}a-Asategui, Restricted Hausdorff spectra of $q$-adic automorphisms, \textit{Adv. Math.}, to appear.
  %
\bibitem{GaGaKl20} O.~Garaialde Oca\~{n}a, A.~Garrido and B.~Klopsch,
  Pro\nobreakdash-$p$ groups of positive rank gradient and Hausdorff
  dimension, \textit{J.\ London Math.\ Soc.} \textbf{101} (2020),
  1008--1040.
  %
\bibitem{GoZo21} J.~Gonz\'{a}lez-S\'{a}nchez and A.~Zozaya, Standard
  Hausdorff spectrum of compact $\F_p[\![t]\!]$-analytic groups,
  \textit{Monatsh.\ Math.} \textbf{195} (2021), 401--419.
  %
\bibitem{Ha76} M.~Hall, Jr., \textit{The theory of groups}, Chelsea
  Publishing Co., New York, 1976.
  %
\bibitem{HeKlTh25} I.~de las Heras, B.~Klopsch and A.~Thillaisundaram, The lower $p$-series of analytic  pro\nobreakdash-$p$
  groups and Hausdorff dimension, arXiv preprint: 2402.06876v2.
%
 \bibitem{HeKl22} I.~de las Heras and B.~Klopsch, A pro\nobreakdash-$p$
  group with full normal Hausdorff spectra, \textit{Math.\ Nachr.}
  \textbf{295} (2022), 89--102.
  %
\bibitem{HeTh22a} I.~de las Heras and A.~Thillaisundaram, A pro-$2$ group
  with full normal Hausdorff spectra, \textit{J.\ Group Theory}
  \textbf{25} (2022), 867--885.
  %
\bibitem{HeTh22b} I.~de las Heras and A.~Thillaisundaram, The finitely
  generated Hausdorff spectra of a family of pro\nobreakdash-$p$ groups,
  \textit{J.\ Algebra} \textbf{606} (2022), 266--297.
  %
\bibitem{JZ08} A.~Jaikin-Zapirain, On the verbal width of finitely
  generated pro\nobreakdash-$p$ groups, \textit{Rev.\ Mat.\ Iberoam.}\
  \textbf{24} (2008), 617--630.
  %
\bibitem{KLP97} G.~Klaas, C.\,R.~Leedham-Green and W.~Plesken,
  \textit{Linear pro\nobreakdash-$p$ groups of finite width}, Springer-Verlag,
  Berlin, 1997.
  %
\bibitem{Kl03} B.~Klopsch, Zeta functions related to the
  pro\nobreakdash-$p$ group $\text{SL}_1(\Delta_p)$, \textit{Math.\
    Proc.\ Camb.\ Phil.\ Soc.}  \textbf{135} (2003), 45--57.
  %
\bibitem{Kl05} B.~Klopsch, On the Lie theory of $p$-adic analytic
  groups, \textit{Math.\ Z.}  \textbf{249} (2005), 713--730.
  %
\bibitem{KlTh19} B.~Klopsch and A.~Thillaisundaram, A
  pro\nobreakdash-$p$ group with infinite normal Hausdorff spectra,
  \textit{Pacific J.\ Math.}  \textbf{303} (2019), 569--603.
%
\bibitem{KlThZu19} B.~Klopsch, A.~Thillaisundaram, and
  A.~Zugadi-Reizabal, Hausdorff dimensions in $p$-adic analytic
  groups, \textit{Israel J.\ Math.} \textbf{231} (2019), 1--23.
  %
\bibitem{La85} J.\,P.~Labute, The determination of the Lie algebra
  associated to the lower central series of a group, \textit{Trans.\
    Amer.\ Math.\ Soc.}  \textbf{288} (1985), 51--57.
  %
\bibitem{La65} M.~Lazard, Groupes analytiques $p$-adiques,
  \textit{Publ.\ Math.\ IH\'ES} \textbf{26} (1965), 389--603.
\end{thebibliography}
\end{document}